\documentclass[10pt,twoside]{article}
\usepackage{mathrsfs}
\usepackage{amssymb}
\usepackage{amsmath}
\usepackage[all]{xy}
\usepackage{amsthm}
\numberwithin{equation}{section}

\newtheorem{theorem}{Theorem}[section]
\newtheorem{proposition}[theorem]{Proposition}
\newtheorem{definition}[theorem]{Definition}

\newtheorem{lemma}[theorem]{Lemma}

\newtheorem{corollary}[theorem]{Corollary}

\newcommand{\Tor}{\operatorname{Tor}}

\newcommand{\Hom}{\operatorname{Hom}}
\newcommand{\Ext}{\operatorname{Ext}}

\newcommand{\WI}{\mathcal{WI}}
\newcommand{\WF}{\mathcal{WF}}

\newcommand{\wid}{\operatorname{wid}}
\newcommand{\wfd}{\operatorname{wfd}}
\newcommand{\ra}{\rightarrow}

\title{ \bf Homological properties of modules with  finite weak injective and weak flat dimensions
}
\vspace{0.2cm}

\author {Tiwei Zhao\thanks{E-mail address: tiweizhao@hotmail.com} \\
{\it \footnotesize Department of Mathematics, Nanjing University, Nanjing 210093, Jiangsu Province, P. R. China}}
\date{ }

\begin{document}

\baselineskip=16pt


\maketitle

\begin{abstract}
 In this paper,  we define a class of  relative derived functors in terms of left or right  weak flat resolutions to compute the weak flat dimension of modules. Moreover, we investigate two classes of modules larger than that of weak injective and weak flat modules, study the existence of covers and preenvelopes, and give some applications.
\end{abstract}

\pagestyle{myheadings}
\markboth{\rightline {\scriptsize   T. Zhao}}
         {\leftline{\scriptsize Homological properties of modules with  finite weak injective and weak flat dimensions}}


\section{Introduction} 

Throughout $R$ is an associative ring with identity and all modules are unitary.  Unless stated otherwise, an $R$-module
will be understood to be a left $R$-module.  We denote by $M^+$ the character module of a module $M$. For unexplained concepts and notations, we refer the
readers to \cite{EJ,Ro}.

In the theory of homological algebra, injective modules play an important and basic role in various aspects, especially for characterizing Noetherian rings. In 1970, to generalize the homological properties from Noetherian rings to coherent rings, Stenstr\"{o}m introduced the notion of FP-injective modules  in \cite{St}. In this process, finitely generated modules should in general be replaced by finitely presented modules.  Recall  that an $R$-module $M$ is called FP-injective (or absolutely pure \cite{Ma}) if $\mbox{Ext}^1_R(L,M)=0$ for every finitely presented  $R$-module $L$, and accordingly, the FP-injective dimension of $M$, denoted by $\mbox{FP-id}_R(M)$, is defined to be the smallest non-negative integer $n$ such that $\mbox{Ext}^{n+1}_R(L,M)=0$ for every finitely presented $R$-module $L$. If no such that $n$ exists, one defines $\mbox{FP-id}_R(M)=\infty$. It is well-known that the FP-injective modules play an important role in characterizing coherent rings, and have nice properties over coherent rings analogous to that of injective modules over Noetherian rings (e.g. \cite{EJ,Ma,Mao,MD,Me,ODL,Pi,St}).

Recently, as a nice generalization of Stenstr\"{o}m's viewpoint, Gao and Wang introduced the notions of weak injective and weak flat modules (\cite{GW}). In this process, finitely presented modules were  replaced by  super finitely
presented modules (see \cite{GW0} or Sec. 2 for the definition). These classes of modules were also investigated by Bravo et al. in \cite{BGH} and, in their paper, these two classes of modules were respectively called absolutely clean and level modules, which come from the ideas of absolutely pure and flat modules, respectively. In \cite{BGH,GH,GW} many results on homological aspect were generalized from coherent rings to arbitrary
rings. So, undoubtedly, these classes of modules will play a crucial role in homological algebra.

Our attempt in this paper is to further investigate the  homological properties of modules with finite weak injective and weak flat dimensions. The structure of this paper is as follows:

In Section 2, we first recall some terminologies and  preliminaries. In \cite{GH}, Gao and Huang investigated the relative derived functors $\mbox{Ext}^{\mathcal{WI}}_n(-,-)$ and $\mbox{Tor}^n_{\mathcal{W}}(-,-)$ in terms of left or right $\mathcal{WI}$-resolution or right $\mathcal{WF}$-resolution to compute the  weak injective dimension of modules.  Here, we introduce another relative derived functor, $\mbox{Ext}^{\mathcal{WF}}_n(-,-)$, in terms of left or right $\mathcal{WF}$-resolutions to compute the weak flat dimension of modules. After introducing these relative derived functors, in the rest of this section, we will establish some isomorphism theorems about the relative (co)homological groups
\begin{gather*}
\mbox{Ext}^{\mathcal{WI}}_n(M,N), \ \mbox{Ext}^{\mathcal{WF}}_n(M,N), \mbox{ and } \mbox{Tor}^n_{\mathcal{W}}(M,N)
\end{gather*}
 and explore their intimate connections (see Proposition $\ref{twoiso}$, $\ref{threeiso}$ and $\ref{111}$).

 In Section 3, as some applications, we characterize the weak flat dimension of modules in terms of the relative derived functor
$ \mbox{Ext}^{\mathcal{WF}}_n(-,-)$
which completes the analogous results of Gao and Huang in \cite{GH}.

In the relative homological algebra,  it has become an active branch of algebra to study the existence of envelopes and covers by different classes of
modules, especially after the appearance
of these concepts in \cite{En} (with the terminology envelopes and covers)
and in \cite{AS} (with the terminology minimal left and right
approximations). For instance, Enochs et al. in \cite{AEGO} studied the existence of envelopes and covers by modules of finite injective and projective dimensions. They showed that every module has an $\mathcal{I}_{\leq n}$-preenvelope (\cite[Prop. 3.1]{AEGO}) and, over a left perfect ring, every module has a $\mathcal{P}_{\leq n}$-cover (\cite[Cor. 4.3]{AEGO}) where $\mathcal{I}_{\leq n}$ and $\mathcal{P}_{\leq n}$ denote the classes of modules of injective
and projective dimensions at most $n$ respectively.  Hence, as a more general and new setting,  in Section 4,  we investigate two classes of modules with richer contents, namely $\WI_{\leq n}$ and $\WF_{\leq n}$ for some non-negative integer $n$. We show that

$\bullet$ $\WI_{\leq n}$ are both covering and preenveloping, and

$\bullet$ $\WF_{\leq n}$  are both covering and preenveloping.\\
Some descriptions of weak injective envelopes and covers (in terms of the left orthogonal class of weak injective modules) are given in the rest of this section. For example, we show that

(1) Every  $R$-module has a weak injective cover with the unique mapping property if and only if  $\operatorname{l.sp.gldim}(R)\leq 2$ (see Proposition $\ref{11111}$), where $\operatorname{l.sp.gldim}(R)$ denotes the left super finitely presented dimension of a ring $R$;

(2) If ${^\perp\WI}$ is closed under pure quotients, then $\WI$ is enveloping, etc.

\section{Weak derived functors of $\mbox{Hom}_R(-,-)$ and $-\otimes_R-$ and their connections}

\subsection{Weak injective, weak flat modules and dimensions}

We first  recall some terminologies and  preliminaries. For more details, we refer the readers to \cite{GH,GW}.

 Recall from \cite{GW0} that an $R$-module $M$ is called \emph{super finitely
presented} if there exists an exact sequence $\cdots \rightarrow
F_1\rightarrow F_0\rightarrow M\rightarrow 0$, where each $F_i$ is
finitely generated and projective. Following this, Gao and Wang gave the
definition of weak injective modules in terms of super finitely
presented modules as follows, which is a proper generalization of the
 FP-injective modules by \cite[Ex. 2.8]{GW}.

\begin{definition}{\rm(\cite[Defs. 2.1 and 3.2]{GW})
An $R$-module $M$ is called \emph{weak injective} if $\mbox{Ext}^1_R(N,M)=0$ for every super finitely presented  $R$-module $N$. A right $R$-module $M$ is called \emph{weak flat} if $\mbox{Tor}_1^R(M,N)=0$ for every super finitely presented  $R$-module $N$.

Accordingly, the \emph{weak injective dimension} of an $R$-module $M$, denoted by $\wid_R(M)$, is defined to be the smallest non-negative integer $n$ such that $\mbox{Ext}^{n+1}_R(N,M)=0$ for every super finitely presented  $R$-module $N$. If no such that $n$ exists, one defines $\wid_R(M)=\infty$. The \emph{weak flat dimension} of a right $R$-module $M$, denoted by $\wfd_R(M)$, is defined to be the smallest non-negative integer $n$ such that $\mbox{Tor}_{n+1}^R(M,N)=0$ for every super finitely presented  $R$-module $N$.  If no such that $n$ exists, one defines $\wfd_R(M)=\infty$.}
\end{definition}

For a given ring $R$, we  denote by $\mathcal{WI}$ and
$\mathcal{WF}$  the classes of weak injective and weak flat modules,
respectively.

Following \cite[Thm. 2.10 and Rem. 2.2(2)]{GW}, an $R$-module $M$ is weak injective if and only if $M^+$ is weak flat; a right $R$-module $M$ is weak flat if and only if $M^+$ is weak injective. More generally, we have

\begin{proposition}\label{lem} Let $R$ and $S$ be rings and $M$ an $S$-$R$-bimodule. Then

$(1)$  $\wid_S(M)=\wfd_S(\operatorname{Hom}_R(M,E))$ for a faithfully injective right $R$-module $E$;

$(2)$  $\wfd_R(M)=\wid_R(\operatorname{Hom}_S(M,E))$ for a faithfully injective left $S$-module $E$.
\end{proposition}

\begin{proof}
(1) For any super finitely presented left $S$-module $L$ and any $n\geq 0$, $\mbox{Tor}^S_n(\mbox{Hom}_R(M,E),L)\cong \mbox{Hom}_R(\mbox{Ext}^n_S(L,M),E)$ by \cite[Lem. 1.2.11]{GT}. So we have
$$
\begin{aligned}
\wid_S(M)\leq n\Leftrightarrow & \mbox{Ext}_S^{n+1}(L,M)=0 \mbox{ for every super finitely presented left $S$-module $L$}\\
\Leftrightarrow & \mbox{Tor}_{n+1}^S(\mbox{Hom}_R(M,E),L)=0  \mbox{ for every super finitely presented left $S$-module $L$}\\
\Leftrightarrow & \wfd_S(\mbox{Hom}_R(M,E))\leq n.
\end{aligned}
$$
Therefore, $\wid_S(M)=\wfd_S(\operatorname{Hom}_R(M,E))$.

$(2)$ By \cite[Lem. 1.2.11]{GT} or \cite[Cor. 10.63]{Ro}, $\mbox{Hom}_S(\mbox{Tor}^R_n(M,L),E)\cong \mbox{Ext}^n_R(L,\mbox{Hom}_S(M,E))$ for any left $R$-module $L$ and any $n\geq 0$. So we have
$$
\begin{aligned}
\wfd_R(M)\leq n\Leftrightarrow & \mbox{Tor}^R_{n+1}(M,L)=0 \mbox{ for every super finitely presented left $R$-module $L$}\\
\Leftrightarrow & \mbox{Ext}^{n+1}_R(L,\mbox{Hom}_S(M,E))=0  \mbox{ for every super finitely presented left $R$-module $L$}\\
\Leftrightarrow & \wid_R(\operatorname{Hom}_S(M,E))\leq n.
\end{aligned}
$$
Therefore, $\wfd_R(M)=\wid_R(\operatorname{Hom}_S(M,E))$.
\end{proof}

\begin{corollary} Let $R$ and $S$ be rings and $M$ an $S$-$R$-bimodule. Then

$(1)$ $M$ is weak injective as a left $S$-module if and only if $\operatorname{Hom}_R(M,E)$ is a weak flat right $S$-module for a faithfully injective right $R$-module $E$;

$(2)$ $M$ is weak flat as a right $R$-module if and only if $\operatorname{Hom}_S(M,E)$ is a weak injective left $R$-module for a faithfully injective left $S$-module $E$.
\end{corollary}

\begin{corollary}\label{cor}
$(1)$ For a left $R$-module $M$, we have $\wid_R(M)=\wfd_R(M^+)$;

$(2)$ For a right $R$-module $M$, we have $\wfd_R(M)=\wid_R(M^+)$.
\end{corollary}

\subsection{Weak derived functors}

\begin{definition}\label{def-env} {\rm(\cite[Def. 6.1.1]{EJ}) Let $\mathcal {C}$ be a class of  $R$-modules. By a
\emph{$\mathcal {C}$-preenvelope} of an $R$-module $M$, one means a morphism
$\varphi: M \rightarrow C$ with $C\in \mathcal {C}$ such that for
any morphism $f: M\rightarrow C'$ with $C'\in \mathcal {C}$,
there exists a morphism $g:C\rightarrow C'$ such that
$g\varphi=f$, that is, there is the following commutative diagram:
\begin{equation}\label{cover diag}
  \xymatrix{
  M \ar[dr]_{f} \ar[r]^{\varphi}
                & C \ar@{.>}[d]^{g}  \\
                & C'             }
\end{equation}
In addition, if when $C'=C$ and $f=\varphi$,
the only such $g$ are automorphisms of $C$, then $\varphi: M
\rightarrow C$ is called a \emph{$\mathcal {C}$-envelope} of $M$.}\end{definition}

Dually, one may give the notion of \emph{$\mathcal {C}$-(pre)cover}
of a module. Note that $\mathcal {C}$-envelopes and $\mathcal
{C}$-covers may not exist in general, but if they exist, they are
unique up to isomorphism (\cite{EJ}).

In what follows, we denote by $_R\mathcal{M}$ (resp.  $\mathcal{M}_R$) the category of left (resp. right) $R$-modules.

Following \cite[Thm. 3.4]{GH}, every  $R$-module has a weak injective preenvelope. So every  $R$-module $M$ has a right $\mathcal{WI}$-resolution, that is, there is a complex $$0\rightarrow M\rightarrow W^0\rightarrow W^1\rightarrow \cdots$$ with each $W^i$ weak injective such that the functor $\mbox{Hom}_R(-,W)$ makes this complex exact for every weak injective  $R$-module $W$. Note that since every injective  $R$-module is weak injective, the above complex is also exact.

On the other hand, every  $R$-module has a weak injective (pre)cover by \cite[Thm. 3.1]{GH}. So every  $R$-module $M$ has a left $\mathcal{WI}$-resolution, that is, there is a complex $$\cdots\rightarrow W_1\rightarrow W_0\rightarrow M\rightarrow 0\mbox{ (not necessarily exact) }$$ with each $W_i$ weak injective such that the functor $\mbox{Hom}_R(W,-)$ makes this complex exact for every weak injective  $R$-module $W$.

Thus, by \cite[Def. 8.2.13]{EJ} or \cite[Sec. 4]{GH}, $\mbox{Hom}_R(-,-)$ is left balanced on $_R\mathcal{M}\times {_R\mathcal{M}}$ by $\mathcal{WI}\times \mathcal{WI}$. Let $$\mbox{Ext}^{\mathcal{WI}}_n(-,-)$$ denote the $n$th left derived functor of $\mbox{Hom}_R(-,-)$ with respect to the pair $\mathcal{WI}\times \mathcal{WI}$. Then for any two  $R$-modules $M$ and $N$, $$\mbox{Ext}^{\mathcal{WI}}_n(M,N)$$ may be computed by using a right $\mathcal{WI}$-resolution of $M$ or a left $\mathcal{WI}$-resolution of $N$ (\cite{GH}).

Following \cite[Sec. 3]{GH} or \cite[Thm. 12]{ET}, every right  $R$-module has a weak flat (pre)cover. So every right $R$-module $M$ has a left $\mathcal{WF}$-resolution, that is, there is a complex $$\cdots\rightarrow W_1\rightarrow W_0\rightarrow M\rightarrow 0$$  with each $W_i$ weak flat such that the functor $\mbox{Hom}_R(W,-)$ makes this complex exact for every weak flat right $R$-module $W$. Since every projective right $R$-module is weak flat, this complex is also exact.

On the other hand, every right $R$-module has a weak flat preenvelope by \cite[Thm. 2.15]{GW}. So every right $R$-module $M$ has a right $\mathcal{WF}$-resolution, that is, there is a complex $$0\rightarrow M\rightarrow W^0\rightarrow W^1\rightarrow \cdots\mbox{  (not necessarily exact) }$$  with each $W^i$ weak flat such that the functor $\mbox{Hom}_R(-,W)$ makes this complex exact for every weak flat right $R$-module $W$.

Similarly, $\mbox{Hom}_R(-,-)$ is left balanced on $\mathcal{M}_R\times \mathcal{M}_R$ by $\mathcal{WF}\times \mathcal{WF}$. Let $$\mbox{Ext}^{\mathcal{WF}}_n(-,-)$$ denote the $n$th left derived functor of $\mbox{Hom}_R(-,-)$ with respect to the pair $\mathcal{WF}\times \mathcal{WF}$. Then for any two right $R$-modules $M$ and $N$, $$\mbox{Ext}^{\mathcal{WF}}_n(M,N)$$ may be computed by using a right $\mathcal{WF}$-resolution of $M$ or a left $\mathcal{WF}$-resolution of $N$.

We also recall from \cite[Sec. 5]{GH} that the functor $-\otimes_R-$ is right balanced on $\mathcal{M}_R\times {_R\mathcal{M}}$ by $\mathcal{WF}\times \mathcal{WI}$. Let $$\mbox{Tor}^n_{\mathcal{W}}(-,-)$$ denote the $n$th right derived functor of $-\otimes_R-$ with respect to $\mathcal{WF}\times \mathcal{WI}$. Then for any right $R$-module $M$ and any  $R$-module $N$, $$\mbox{Tor}^n_{\mathcal{W}}(M,N)$$ may be computed by using a right $\mathcal{WF}$-resolution of $M$ or a right $\mathcal{WI}$-resolution of $N$.

In the following, we will establish some isomorphism theorems about the relative (co)homological groups $$\mbox{Ext}^{\mathcal{WI}}_n(M,N), \  \mbox{Ext}^{\mathcal{WF}}_n(M,N), \mbox{ and } \mbox{Tor}^n_{\mathcal{W}}(M,N),$$ and explore their intimate connections.

\begin{proposition}\label{twoiso}
Let $R$ and $S$ be rings.

$(1)$ If $L$ is a right $R$-module, $M$ an $R$-$S$-bimodule and $N$ an injective right $S$-module, then we have
$$
\operatorname{Ext}^{\mathcal{WF}}_n(L,\operatorname{Hom}_S(M,N))\cong \operatorname{Hom}_S(\operatorname{Tor}^n_{\mathcal{W}}(L,M),N)
$$
for   any non-negative integer $n$.

$(2)$ If $L$ is an $S$-$R$-bimodule, $M$ a left $R$-module and $N$ an injective left $S$-module, then we have
$$
\operatorname{Ext}^{\mathcal{WI}}_n(M,\operatorname{Hom}_S(L,N))\cong \operatorname{Hom}_S(\operatorname{Tor}^n_{\mathcal{W}}(L,M),N)
$$
for   any non-negative integer $n$.
\end{proposition}

\begin{proof}
(1) Let $\mathscr{W}^{\mathcal{F}}_L:0\rightarrow W^0\rightarrow
W^1\rightarrow W^2\rightarrow \cdots$ be a deleted right
$\mathcal{WF}$-resolution of $L$. Then we have the following complex
$$
\mathscr{W}^{\mathcal{F}}_L\otimes_RM:0\rightarrow
W^0\otimes_RM\rightarrow W^1\otimes_RM\rightarrow
W^2\otimes_RM\rightarrow \cdots.
$$
Since $N$ is an injective right $S$-module, the functor $\mbox{Hom}_S(-,N)$ is exact. Hence we have the following isomorphisms:
$$
\begin{array}{cc}
\operatorname{Hom}_S(\operatorname{Tor}^n_{\mathcal{W}}(L,M),N)&\!\!\!\!\!\!\!\!\!\!\!\!\!\!\!\!= \operatorname{Hom}_S(\operatorname{H}^n(\mathscr{W}^{\mathcal{F}}_L\otimes_RM),N)\\
&\!\!\!\!\!\!\!\!\!\!\!\!\!\!\!\!\cong \mbox{H}_n(\mbox{Hom}_S(\mathscr{W}^{\mathcal{F}}_L\otimes_RM,N))\\
&\!\!\!\cong \mbox{H}_n(\mbox{Hom}_R(\mathscr{W}^{\mathcal{F}}_L,\mbox{Hom}_S(M,N)))\\
&\!\!\!\!\!\!\!\!\!\!\!\!\!\!\!\!\!\!\!= \operatorname{Ext}^{\mathcal{WF}}_n(L,\operatorname{Hom}_S(M,N)).
\end{array}
$$

(2) Let $\mathscr{W}^{\mathcal{I}}_M:0\rightarrow W^0\rightarrow
W^1\rightarrow W^2\rightarrow \cdots$ be a deleted right
$\mathcal{WI}$-resolution of $M$. Then we have the following complex
$$
L\otimes_R\mathscr{W}^{\mathcal{I}}_M:0\rightarrow L\otimes_R W^0
\rightarrow L\otimes_R W^1 \rightarrow L\otimes_R W^2 \rightarrow
\cdots.
$$
Since $N$ is an injective left $S$-module, the functor $\mbox{Hom}_S(-,N)$ is exact. Hence we have the following isomorphisms:
$$
\begin{array}{cc}
\operatorname{Hom}_S(\operatorname{Tor}^n_{\mathcal{W}}(L,M),N)&\!\!\!\!\!\!\!\!\!\!\!\!\!\!\!\!= \operatorname{Hom}_S(\operatorname{H}^n(L\otimes_R\mathscr{W}^{\mathcal{I}}_M),N)\\
&\!\!\!\!\!\!\!\!\!\!\!\!\!\!\!\!\cong \mbox{H}_n(\mbox{Hom}_S(L\otimes_R\mathscr{W}^{\mathcal{I}}_M,N))\\
&\!\!\!\cong \mbox{H}_n(\mbox{Hom}_R(\mathscr{W}^{\mathcal{I}}_M,\mbox{Hom}_S(L,N)))\\
&\!\!\!\!\!\!\!\!\!\!\!\!\!\!\!\!\!\!\!= \operatorname{Ext}^{\mathcal{WI}}_n(M,\operatorname{Hom}_S(L,N)).
\end{array}
$$
\end{proof}

By setting $S=\mathbb{Z}$ and $N=\mathbb{Q}/\mathbb{Z}$ in Proposition $\ref{twoiso}$, we have the following corollary.

\begin{corollary}\label{twoisocor}
Let $M$ be a right $R$-module and $N$ a left $R$-module. Then
$$
\operatorname{Ext}^{\mathcal{WF}}_n(M,N^+)\cong \operatorname{Tor}^n_{\mathcal{W}}(M,N)^+\cong \operatorname{Ext}^{\mathcal{WI}}_n(N,M^+)
$$
for any non-negative integer $n$.
\end{corollary}

\begin{proposition}\label{threeiso}
Let $R$ and $S$ be rings.

$(1)$ If $L$ is an $R$-$S$-bimodule, $M$ a projective left $S$-module and $N$ a left $R$-module, then
$$
\operatorname{Ext}^{\mathcal{WI}}_n(L\otimes_SM,N)\cong \operatorname{Hom}_S(M,\operatorname{Ext}^{\mathcal{WI}}_n(L,N))
$$
for any non-negative integer $n$.

$(2)$ If $L$ is an $S$-$R$-bimodule, $M$ a projective right $S$-module and $N$ a right $R$-module, then
$$
\operatorname{Ext}^{\mathcal{WF}}_n(M\otimes_SL,N)\cong \operatorname{Hom}_S(M,\operatorname{Ext}^{\mathcal{WF}}_n(L,N))
$$
for any non-negative integer $n$.

$(3)$ If $L$ is a flat right $S$-module, $M$ an $S$-$R$-bimodule and $N$ a left $R$-module, then
$$
\operatorname{Tor}^n_{\mathcal{W}}(L\otimes_SM,N)\cong L\otimes_S\operatorname{Tor}^n_{\mathcal{W}}(M,N)
$$
for any non-negative integer $n$.
\end{proposition}

\begin{proof}
(1) Let $\widehat{\mathscr{W}}^{\mathcal{I}}_N:\cdots \rightarrow
W_2\rightarrow W_1\rightarrow W_0\rightarrow 0$ be a deleted left
$\mathcal{WI}$-resolution of $N$. Since  $M$ is a projective left
$S$-module, the functor $\mbox{Hom}_S(M,-)$ is exact, and hence we
have the following isomorphisms:
$$
\begin{array}{cc}\mbox{Ext}^{\mathcal{WI}}_n(L\otimes_SM,N)&\!\!\!\!\!\!\!\!\!\!\!\!\!\!\!\!=\mbox{H}_n(\mbox{Hom}_R(L\otimes_SM,\widehat{\mathscr{W}}^{\mathcal{I}}_N))\\
&\!\!\!\cong\mbox{H}_n(\mbox{Hom}_S(M,\mbox{Hom}_R(L,\widehat{\mathscr{W}}^{\mathcal{I}}_N)))\\
&\!\!\!\cong\mbox{Hom}_S(M,\mbox{H}_n(\mbox{Hom}_R(L,\widehat{\mathscr{W}}^{\mathcal{I}}_N)))\\
&\!\!\!\!\!\!\!\!\!\!\!\!\!\!\!\!\!\!=\operatorname{Hom}_S(M,\operatorname{Ext}^{\mathcal{WI}}_n(L,N)).\end{array}
$$

(2) Let $\widehat{\mathscr{W}}^{\mathcal{F}}_N:\cdots \rightarrow
W_2\rightarrow W_1\rightarrow W_0\rightarrow 0$ be a deleted left
$\mathcal{WF}$-resolution of $N$. Since  $M$ is a projective left
$S$-module, the functor $\mbox{Hom}_S(M,-)$ is exact, and hence we
have the following isomorphisms:
$$
\begin{array}{cc}\mbox{Ext}^{\mathcal{WF}}_n(M\otimes_SL,N)&\!\!\!\!\!\!\!\!\!\!\!\!\!\!\!\!=\mbox{H}_n(\mbox{Hom}_R(M\otimes_SL,\widehat{\mathscr{W}}^{\mathcal{F}}_N))\\
&\!\!\!\cong\mbox{H}_n(\mbox{Hom}_S(M,\mbox{Hom}_R(L,\widehat{\mathscr{W}}^{\mathcal{F}}_N)))\\
&\!\!\!\cong\mbox{Hom}_S(M,\mbox{H}_n(\mbox{Hom}_R(L,\widehat{\mathscr{W}}^{\mathcal{F}}_N)))\\
&\!\!\!\!\!\!\!\!\!\!\!\!\!\!\!\!\!\!=\operatorname{Hom}_S(M,\operatorname{Ext}^{\mathcal{WF}}_n(L,N)).\end{array}
$$

(3) Let $\mathscr{W}^{\mathcal{I}}_N:0\rightarrow W^0\rightarrow
W^1\rightarrow W^2\rightarrow \cdots$ be a deleted right
$\mathcal{WI}$-resolution of $N$. Since $L$ is a flat right
$S$-module, the functor $L\otimes_S-$ is exact, and hence we have
the following isomorphisms:
$$
\begin{array}{cc}L\otimes_S\mbox{Tor}_{\mathcal{W}}^n(M,N)&\!\!\!\!\!\!\!\!=L\otimes_S\mbox{H}^n(M\otimes_R\mathscr{W}^{\mathcal{I}}_N)\\
&\!\!\!\cong\mbox{H}^n(L\otimes_S(M\otimes_R\mathscr{W}^{\mathcal{I}}_N))\\
&\!\!\!\cong\mbox{H}^n((L\otimes_SM)\otimes_R\mathscr{W}^{\mathcal{I}}_N)\\
&\!\!\!\!\!\!\!\!\!\!\!\!\!\!=\operatorname{Tor}^n_{\mathcal{W}}(L\otimes_SM,N).\end{array}
$$
\end{proof}

\begin{proposition}\label{111}
Let $R$ be a commutative ring and $M$ a projective $R$-module. For any $R$-modules $L$ and $N$, we have

$(1)$ $$
\operatorname{Ext}^{\mathcal{WI}}_n(L\otimes_RM,N)\cong \operatorname{Hom}_R(M,\operatorname{Ext}^{\mathcal{WI}}_n(L,N))\cong \operatorname{Ext}^{\mathcal{WI}}_n(L,\operatorname{Hom}_R(M,N));
$$

$(2)$
$$
\operatorname{Ext}^{\mathcal{WF}}_n(M\otimes_RL,N)\cong \operatorname{Hom}_R(M,\operatorname{Ext}^{\mathcal{WF}}_n(L,N))\cong \operatorname{Ext}^{\mathcal{WF}}_n(L,\operatorname{Hom}_R(M,N))
$$
for  any non-negative integer $n$.
\end{proposition}

\begin{proof}
(1) The first isomorphism follows immediately from Proposition
$\ref{threeiso}$(1). On the other hand, let
$\widehat{\mathscr{W}}^{\mathcal{I}}:\cdots \rightarrow W_1\rightarrow
W_0\rightarrow N\rightarrow 0$ be a  left $\mathcal{WI}$-resolution
of $N$. For any super finitely presented $R$-module $F$, since $M$
is  projective, we may obtain that $\mbox{Hom}_R(M,W_i)$ is weak
injective from the following isomorphisms:
$$
\begin{array}{cc}\mbox{Ext}^n_R(F,\mbox{Hom}_R(M,W_i))&\!\!\!\!\!\!\!\!\!\!\!\!\!\!\!\!\!\cong \mbox{Ext}^n_R(M\otimes_RF,W_i)\\ &\!\!\!\!\!\!\!\!\!\!\!\!\!\!\!\!\!\cong \mbox{Ext}^n_R(F\otimes_RM,W_i)\\&\!\!\!\cong \mbox{Hom}_R(M,\mbox{Ext}^n_R(F,W_i)),\end{array}
$$
where the first isomorphism comes from \cite[Cor. 10.65]{Ro} and the third comes from \cite[p. 668]{Ro}.
Meanwhile, the complex
$\mbox{Hom}_R(W,\mbox{Hom}_R(M,\widehat{\mathscr{W}}^{\mathcal{I}}))\cong
\mbox{Hom}_R(M,\mbox{Hom}_R(W,\widehat{\mathscr{W}}^{\mathcal{I}}))$ is
exact for any weak injective $R$-module $W$. So the complex
$\mbox{Hom}_R(M,\widehat{\mathscr{W}}^{\mathcal{I}})$ is a left
$\mathcal{WI}$-resolution of $\mbox{Hom}_R(M,N)$. Thus we have the
following isomorphisms:
$$
\begin{array}{cc}
\operatorname{Ext}^{\mathcal{WI}}_n(L,\operatorname{Hom}_R(M,N))&\!\!\!=\mbox{H}_n(\mbox{Hom}_R(L,\mbox{Hom}_R(M,\widehat{\mathscr{W}}^{\mathcal{I}}_N)))\\
&\!\!\!\!\!\!\!\!\!\!\!\!\!\!\cong \mbox{H}_n(\mbox{Hom}_R(L\otimes_RM,\widehat{\mathscr{W}}^{\mathcal{I}}_N))\\
&\!\!\!\!\!\!\!\!\!\!\!\!\!\!\!\!\!\!\!\!\!\!\!\!\!\!\!\!\!=\operatorname{Ext}^{\mathcal{WI}}_n(L\otimes_RM,N).
\end{array}
$$

(2) The first isomorphism follows immediately from Proposition
$\ref{threeiso}$(2). On the other hand, let
$\widehat{\mathscr{W}}^{\mathcal{F}}:\cdots \rightarrow W_1\rightarrow
W_0\rightarrow N\rightarrow 0$ be a  left $\mathcal{WF}$-resolution
of $N$. Since $M$ is  projective, it is easy to verify that each
$\mbox{Hom}_R(M,W_i)$ is weak flat, and  the complex
$$\mbox{Hom}_R(W,\mbox{Hom}_R(M,\widehat{\mathscr{W}}^{\mathcal{F}}))\cong
\mbox{Hom}_R(M,\mbox{Hom}_R(W,\widehat{\mathscr{W}}^{\mathcal{F}}))$$ is
exact for any weak flat $R$-module $W$. So the complex
$\mbox{Hom}_R(M,\widehat{\mathscr{W}}^{\mathcal{F}})$ is a left
$\mathcal{WF}$-resolution of $\mbox{Hom}_R(M,N)$. Thus we have the
following isomorphisms:
$$
\begin{array}{cc}
\operatorname{Ext}^{\mathcal{WF}}_n(L,\operatorname{Hom}_R(M,N))&\!\!\!=\mbox{H}_n(\mbox{Hom}_R(L,\mbox{Hom}_R(M,\widehat{\mathscr{W}}^{\mathcal{F}}_N)))\\
&\!\!\!\!\!\!\!\!\!\!\!\!\!\!\cong \mbox{H}_n(\mbox{Hom}_R(M\otimes_RL,\widehat{\mathscr{W}}^{\mathcal{F}}_N))\\
&\!\!\!\!\!\!\!\!\!\!\!\!\!\!\!\!\!\!\!\!\!\!\!\!\!\!\!\!\!=\operatorname{Ext}^{\mathcal{WI}}_n(M\otimes_RL,N).
\end{array}
$$
\end{proof}

\section{The homological properties of   modules  with finite weak flat   dimensions}

This section is devoted to investigating further the  weak flat modules and dimension in terms of the relative derived functors $  \mbox{Ext}^{\mathcal{WF}}_n(-,-) \mbox{ and } \mbox{Tor}^n_{\mathcal{W}}(-,-).$

Note that  for any two right $R$-modules $M$ and $N$, there exists a canonical map $$\theta: \mbox{Ext}^{\mathcal{WF}}_0(M,N)\rightarrow \mbox{Hom}_R(M,N)$$
and for any right $R$-module $M$ and any  $R$-module $L$, there exists a canonical map $$\vartheta: M\otimes_RL\rightarrow \mbox{Tor}_{\mathcal{F}}^0(M,L).$$

 We recall from \cite{EJ} that a short exact sequence $0\rightarrow L\rightarrow M\rightarrow N\rightarrow 0$ of $R$-modules is \emph{pure exact} if for every right $R$-module $Q$, the induced sequence $0\rightarrow Q\otimes_RL\rightarrow Q\otimes_RM\rightarrow Q\otimes_RN\rightarrow 0$ is exact, and in this case, the morphism $L\rightarrow M$ (resp. $M\rightarrow N$) is called a  \emph{pure  monomorphism} (resp. a \emph{pure epimorphism}).  An $R$-module $P$ is said to be \emph{pure projective} (resp. \emph{pure injective}) if the functor $\operatorname{Hom}_R(P,-)$ (resp. $\operatorname{Hom}_R(-,P)$) is exact with respect to all short pure exact sequences.

 By \cite[Prop. 4.4]{GH},  an $R$-module $M$ is weak injective if and only if the map $ \operatorname{Ext}^{\mathcal{WI}}_0(M,N)\rightarrow \operatorname{Hom}_R(M,N)$ is an epimorphism for every  $R$-module $N$. Also, we have

\begin{lemma}\label{wirc}
The following are equivalent for an $R$-module $M$:

$(1)$ $M$ is weak injective;

$(2)$ The map $\sigma: \operatorname{Ext}^{\mathcal{WI}}_0(L,M)\rightarrow \operatorname{Hom}_R(L,M)$ is an epimorphism for every   $R$-module $L$;

$(3)$ The map $\sigma: \operatorname{Ext}^{\mathcal{WI}}_0(L,M)\rightarrow \operatorname{Hom}_R(L,M)$ is an epimorphism for every pure projective  $R$-module $L$.
\end{lemma}

\begin{proof}
(1) $\Rightarrow$ (2) $\Rightarrow$  (3)  are trivial.

(3) $\Rightarrow$ (1). Let $\xymatrix@C=0.5cm{
   \cdots\ar[r] &  W_1\ar[r]^{d_1}& W_0\ar[r]^{d_0}& M\ar[r]& 0}$ be a left $\mathcal{WI}$-resolution of $M$ and $L$ a pure projective  $R$-module. Then we have the following deleted  complex  $$\xymatrix@C=0.5cm{
     \cdots \ar[r] & \mbox{Hom}_R(L,W_1) \ar[r]^{{d_1}_\ast} & \mbox{Hom}_R(L,W_0) \ar[r] & 0 }.$$
 By definition, $\operatorname{Ext}^{\mathcal{WI}}_0(L,M)=\mbox{Hom}_R(L,W_0)/\mbox{Im}{d_1}_\ast$. So we have the following commutative diagram:
 $$
 \xymatrix{\mbox{Hom}_R(L,W_0)\ar[rr]^{{d_0}_\ast}\ar[rd]&&\mbox{Hom}_R(L,M).\\
 &\operatorname{Ext}^{\mathcal{WI}}_0(L,M)\ar[ru]_{\sigma}\ar[rd]&\\&&0}
 $$
Thus ${d_0}_\ast$ is an epimorphism by hypothesis, and hence $d_0$ is a pure epimorphism. That is, the sequence $0\rightarrow \mbox{Ker}d_0\rightarrow W_0\stackrel{d_0}\rightarrow M\rightarrow 0$ is pure exact. So $\mbox{Ker}d_0$ is weak injective by \cite[Prop. 2.9]{GW}. Moreover, it is easy to verify  that $M$ is also weak injective.
\end{proof}

\begin{proposition}\label{propweakflat}
The following are equivalent for a right $R$-module $M$:

$(1)$ $M$ is weak flat;

$(2)$ The map $\theta: \operatorname{Ext}^{\mathcal{WF}}_0(L,M)\rightarrow \operatorname{Hom}_R(L,M)$ is an epimorphism for every  right $R$-module $L$;

$(3)$ The map $\theta: \operatorname{Ext}^{\mathcal{WF}}_0(L,M)\rightarrow \operatorname{Hom}_R(L,M)$ is an epimorphism for every pure projective right $R$-module $L$;

$(4)$ The map $\theta: \operatorname{Ext}^{\mathcal{WF}}_0(M,N)\rightarrow \operatorname{Hom}_R(M,N)$ is an epimorphism for every right $R$-module $N$;

$(5)$ The map $\theta: \operatorname{Ext}^{\mathcal{WF}}_0(M,N)\rightarrow \operatorname{Hom}_R(M,N)$ is an epimorphism for every pure injective right $R$-module $N$;

$(6)$  The map $\theta: \operatorname{Ext}^{\mathcal{WF}}_0(M,M)\rightarrow \operatorname{Hom}_R(M,M)$ is an epimorphism;

$(7)$ The map $\vartheta: M\otimes_RL\rightarrow \operatorname{Tor}_{\mathcal{W}}^0(M,L)$ is a monomorphism for every   $R$-module $L$;

$(8)$ The map $\vartheta: M\otimes_RL\rightarrow \operatorname{Tor}_{\mathcal{W}}^0(M,L)$ is a monomorphism for every pure projective  $R$-module $L$.
\end{proposition}

\begin{proof} (1) $\Rightarrow$ (2) $\Rightarrow$ (3), (1) $\Rightarrow$ (4) $\Rightarrow$ (5), (2) $\Rightarrow$  (6) and  (7) $\Rightarrow$  (8) are trivial.

(1) $\Leftrightarrow$ (7) follows from \cite[Prop. 5.4]{GH}.

(3) $\Rightarrow$ (1). Let $\xymatrix@C=0.5cm{
  \cdots \ar[r] & W_1 \ar[r]^{d_1} & W_0 \ar[r]^{d_0} & M \ar[r] & 0 }$ be a left $\mathcal{WF}$-resolution of $M$ and $L$ a pure projective right $R$-module. Then, by hypothesis, it is easy to verify that $\mbox{Hom}_R(L,W_0)\stackrel{{d_0}_\ast}\rightarrow \mbox{Hom}_R(L,M)\rightarrow 0$ is exact, and hence $d_0$ is a pure epimorphism. That is, the sequence $0\rightarrow \mbox{Ker}d_0\rightarrow W_0\rightarrow M\rightarrow 0$ is pure exact. So $\mbox{Ker}d_0$ is weak flat by \cite[Prop. 2.9]{GW} and hence  $M$ is weak flat.

  (5) $\Rightarrow$ (1). Let $\xymatrix@C=0.5cm{
   0\ar[r] & M \ar[r]^{d^0} & W^0 \ar[r]^{d^1} & W^1 \ar[r] & \cdots }$ be a right $\mathcal{WF}$-resolution of $M$ and $N$ a pure injective right $R$-module. Then, by hypothesis, it is easy to verify that $\mbox{Hom}_R(W^0,N)\stackrel{{d^0}^\ast}\rightarrow \mbox{Hom}_R(M,N)\rightarrow 0$ is exact, and hence $d^0$ is a pure monomorphism. Thus $M$ is weak flat by \cite[Prop. 2.9]{GW}.

(6) $\Rightarrow$ (1). Let $f:W_0\rightarrow M$ be a weak flat precover of $M$. By hypothesis, it is easy to verify that there exists $g:M\rightarrow W_0$ such that $fg=1_M$, that is, $f$ is a retraction. Thus $M$ is weak flat as a direct summand of $W_0$.

 (8) $\Rightarrow$ (1).  For any pure projective  $R$-module $L$, since  $0\rightarrow M\otimes_RL\stackrel{\vartheta}\rightarrow \operatorname{Tor}_{\mathcal{W}}^0(M,L)$ is exact, the induced sequence $\operatorname{Tor}_{\mathcal{W}}^0(M,L)^+\stackrel{\vartheta^+}\rightarrow (M\otimes_RL)^+\rightarrow 0$ is also exact. Consider the following commutative diagram:
    $$
    \xymatrix{\operatorname{Tor}_{\mathcal{W}}^0(M,L)^+\ar[r]^{\vartheta^+}\ar[d]^f& (M\otimes_RL)^+\ar[d]^g\ar[r]&0\\
    \mbox{Ext}^{\mathcal{WI}}_0(L,M^+) \ar[r]&\mbox{Hom}_R(L,M^+).&
    }
    $$
Note that $g$ is an isomorphism by adjoint isomorphism, and $f$ is also an isomorphism by Corollary $\ref{twoisocor}$. So the sequence $\mbox{Ext}^{\mathcal{WI}}_0(L,M^+) \rightarrow\mbox{Hom}_R(L,M^+)\rightarrow 0$ is exact. It follows from Lemma \ref{wirc} that $M^+$ is weak injective, and hence $M$ is weak flat by \cite[Rem. 2.2(2)]{GW}.
\end{proof}

\begin{proposition}\label{wfd1}
The following are equivalent for a right $R$-module $M$:

$(1)$ $\wfd_R(M)\leq 1$;

$(2)$ The map $\theta: \operatorname{Ext}^{\mathcal{WF}}_0(L,M)\rightarrow \operatorname{Hom}_R(L,M)$ is a monomorphism for every  right $R$-module $L$;

$(3)$ The map $\theta: \operatorname{Ext}^{\mathcal{WF}}_0(L,M)\rightarrow \operatorname{Hom}_R(L,M)$ is a monomorphism for every pure projective right $R$-module $L$.
\end{proposition}

\begin{proof}
(1) $\Rightarrow$ (2). By hypothesis, $M$ has a left $\mathcal{WF}$-resolution $0\rightarrow W_1\rightarrow W_0\rightarrow M\rightarrow 0$. Applying the functor $\mbox{Hom}_R(L,-)$ to it, we have the following exact sequence
$$
0\rightarrow \mbox{Hom}_R(L,W_1)\rightarrow \mbox{Hom}_R(L,W_0)\rightarrow \mbox{Hom}_R(L,M).
$$
By definition, $\operatorname{Ext}^{\mathcal{WF}}_0(L,M)=\mbox{Hom}_R(L,W_0)\diagup\mbox{Hom}_R(L,W_1)$. Consequently, $\operatorname{Ext}^{\mathcal{WF}}_0(L,M)\rightarrow \operatorname{Hom}_R(L,M)$ is a monomorphism.

(2) $\Rightarrow$ (3) is trivial.

(3) $\Rightarrow$ (1). Let  $f:W_0\rightarrow M$ be a weak flat cover of $M$ and $K=\mbox{Ker}f$. Then we have an exact sequence $0\rightarrow K\rightarrow W_0\rightarrow M\rightarrow 0$ with $W_0$ weak flat. So it suffices to show that $K$ is also weak flat. In fact, the sequence $0\rightarrow K\rightarrow W_0\rightarrow M\rightarrow 0$ is $\mbox{Hom}_R(\mathcal{WF},-)$-exact, so there is the following long exact sequence by \cite[Thm. 8.2.3]{EJ}
$$
\cdots \rightarrow \mbox{Ext}^{\mathcal{WF}}_1(L,M)\rightarrow \mbox{Ext}^{\mathcal{WF}}_0(L,K)\rightarrow \mbox{Ext}^{\mathcal{WF}}_0(L,W_0)\rightarrow \mbox{Ext}^{\mathcal{WF}}_0(L,M)\rightarrow 0
$$
for every pure projective right $R$-module $L$.

Consider the following commutative diagram:
$$
\xymatrix{
&\mbox{Ext}^{\mathcal{WF}}_0(L,K)\ar[r]\ar[d]^{\sigma_K}& \mbox{Ext}^{\mathcal{WF}}_0(L,W_0)\ar[r]\ar[d]^{\sigma_{W_0}}& \mbox{Ext}^{\mathcal{WF}}_0(L,M)\ar[r]\ar[d]^{\sigma_M}& 0\\
0\ar[r]&\mbox{Hom}_R(L,K)\ar[r]&\mbox{Hom}_R(L,W_0)\ar[r]&\mbox{Hom}_R(L,M).&
}
$$
Note that $\sigma_{W_0}$ is an epimorphism by Proposition $\ref{propweakflat}$ and $\sigma_M$ is a monomorphism by assumption. So $\sigma_K$ is an epimorphism by Snake Lemma. Thus $K$ is weak flat by Proposition $\ref{propweakflat}$, and hence $\wfd_R(M)\leq 1$.
\end{proof}

\begin{proposition}\label{charwfd}
Let $n\geq 2$. Then the following are equivalent for a right $R$-module $M$:

$(1)$ $\wfd_R(M)\leq n$;

$(2)$ $\operatorname{Ext}^{\mathcal{WF}}_{n+k}(L,M)=0$ for every right $R$-module $L$ and every $k\geq -1$;

$(3)$ $\operatorname{Ext}^{\mathcal{WF}}_{n-1}(L,M)=0$ for every right $R$-module $L$;

$(4)$ $\operatorname{Ext}^{\mathcal{WF}}_{n+k}(L,M)=0$ for every pure projective right $R$-module $L$ and every $k\geq -1$;

$(5)$ $\operatorname{Ext}^{\mathcal{WF}}_{n-1}(L,M)=0$ for every pure projective right $R$-module $L$;

$(6)$ $\operatorname{Ext}^{\mathcal{WF}}_{n+k}(L,M)=0$ for every finitely presented right $R$-module $L$ and every $k\geq -1$;

$(7)$ $\operatorname{Ext}^{\mathcal{WF}}_{n-1}(L,M)=0$ for every finitely presented  right $R$-module $L$.
\end{proposition}

\begin{proof}
(1) $\Rightarrow$ (2). By hypothesis, we may take a left $\mathcal{WF}$-resolution of $M$ as follows:
$$
\xymatrix@C=0.5cm{
  0 \ar[r] & W_n \ar[r]^{} & W_{n-1}\ar[r]&\cdots \ar[r]^{} & W_1 \ar[r]^{} & W_0 \ar[r]^{} & M \ar[r] & 0. }
$$
Clearly, $\operatorname{Ext}^{\mathcal{WF}}_{n+k}(L,M)=0$ for all right $R$-module $L$ and all $k\geq 1$. Moreover, from the exactness of the following sequence
$$
\xymatrix@C=0.5cm{
  0 \ar[r] & \mbox{Hom}_R(L,W_n) \ar[r]^{} & \mbox{Hom}_R(L,W_{n-1}) \ar[r] & \mbox{Hom}_R(L,W_{n-2}) }
$$
we may obtain that $\operatorname{Ext}^{\mathcal{WF}}_{n-1}(L,M)=0=\operatorname{Ext}^{\mathcal{WF}}_{n}(L,M)$, as desired.

(2) $\Rightarrow$ (3)  $\Rightarrow$ (5)  $\Rightarrow$ (7) and (2) $\Rightarrow$ (4)  $\Rightarrow$ (6)  $\Rightarrow$ (7) are trivial.

(7) $\Rightarrow$ (1). Let $$
\xymatrix{
  \cdots \ar[r] & W_n \ar[r]^{d_n} & W_{n-1}\ar[r]^{d_{n-1}}&\cdots \ar[r]^{} & W_1 \ar[r]^{d_1} & W_0 \ar[r]^{d_0} & M \ar[r] & 0 }
$$
be a left $\mathcal{WF}$-resolution of $M$ and $K_n=\mbox{Im}d_n$. Then we have the following commutative diagram:
$$
\xymatrix{
  \cdots \ar[r] & W_n \ar[rr]^{d_n}\ar[rd]^{\pi} && W_{n-1}\ar[r]^{d_{n-1}}&\cdots  \ar[r]^{d_1} & W_0 \ar[r]^{d_0} & M \ar[r] & 0. \\
 &&K_n\ar[ru]^{i}\ar[rd] &&&&&&\\
 &0\ar[ru]&&0&&&&&
  }
$$
In order to prove that $\wfd_R(M)\leq n$,  we only suffices to show that $K_n$ is weak flat. Note that the class of weak flat $R$-modules is closed under pure quotients. So we may get the desired result just  by proving that  $\pi$ is a pure epimorphism.

By assumption, $\operatorname{Ext}^{\mathcal{WF}}_{n-1}(L,M)=0$ for every finitely presented  right $R$-module $L$, which means that the sequence
$$
\mbox{Hom}_R(L,W_n)\stackrel{{d_n}_\ast}\rightarrow \mbox{Hom}_R(L,W_{n-1})\stackrel{{d_{n-1}}_\ast}\rightarrow \mbox{Hom}_R(L,W_{n-2})
$$
is exact. But the sequence
$$0\rightarrow \mbox{Hom}_R(L,K_n)\stackrel{i_\ast}\rightarrow \mbox{Hom}_R(L,W_{n-1})\stackrel{{d_{n-1}}_\ast}\rightarrow \mbox{Hom}_R(L,W_{n-2})$$ is exact. In the following, we will prove that the sequence $\mbox{Hom}_R(L,W_n)\stackrel{\pi_\ast}\rightarrow \mbox{Hom}_R(L,K_n)\rightarrow 0$ is exact for any finitely presented  right $R$-module $L$. Take any $\alpha\in \mbox{Hom}_R(L,K_n)$. Since ${d_{n-1}}_\ast i_\ast(\alpha)=0$, we have $i\alpha=i_\ast(\alpha)\in \mbox{Ker}{d_{n-1}}_\ast=\mbox{Im}{d_n}_\ast$. Thus there exists $\beta\in\mbox{Hom}_R(L,W_n)$ such that $i\alpha={d_n}_\ast(\beta)={d_n}\beta=i\pi\beta$. Moreover, since $i$ is monic, we have $\alpha=\pi\beta=\pi_\ast(\beta)$. So $\pi_\ast$ is epic, and hence $\pi$ is pure epic, as desired.
\end{proof}

\section{Covers and preenvelopes by modules of finite weak injective and weak flat dimensions}

\subsection{}

In this section, we will investigate two classes of modules, namely modules of weak injective dimension at most $n$ and that of weak flat dimension at most $n$ respectively, and  prove the existence of the corresponding covers and preenvelopes. We mainly use a result of Holm and J{\o}rgensen in \cite{HJ09}.

\begin{definition}{\rm(\cite[Def. 2.1]{HJ09})
A{\emph{ dual pair}} over a ring $R$ is a pair $(\mathcal{C},\mathcal{D})$, where $\mathcal{C}$ is a class of left $R$(resp. right $R$)-modules and $\mathcal{D}$ is a class of right $R$(resp. left $R$)-modules, subject to the following conditions:

(1) For any module $M$, one has  $M\in \mathcal{C}$ if and only if $M^+\in \mathcal{D}$;

(2) $\mathcal{D}$ is closed under direct summands and finite direct sums.

A duality pair $(\mathcal{C},\mathcal{D})$ is called \emph{(co)product-closed} if $\mathcal{C}$ is closed under (co)products.}
\end{definition}

The dual pair plays an important role in the aspect of showing the existence of covers and envelopes. The main theorem obtained by  Holm and J{\o}rgensen (i.e. \cite[Thm. 3.1]{HJ09}) is as follows:

Let $(\mathcal{C},\mathcal{D})$ be a dual pair. Then $\mathcal{C}$  is closed under pure submodules, pure quotients and pure extensions. Furthermore, the following hold:

(a) If $(\mathcal{C},\mathcal{D})$ is product-closed, then $\mathcal{C}$ is preenveloping;

(b) If $(\mathcal{C},\mathcal{D})$ is coproduct-closed, then $\mathcal{C}$ is covering;

(c) If $(\mathcal{C},\mathcal{D})$ is perfect, then  $(\mathcal{C},\mathcal{C}^\perp)$ is a perfect cotorsion pair.

\medskip

Let $n$ be a fixed non-negative integer. We denote by $\WI_{\leq n}$ and $\WF_{\leq n}$ the classes of modules with weak injective dimension and weak flat dimension at most $n$ respectively.

\begin{proposition}\label{dual pair 1}
The pair $(\WF_{\leq n},\WI_{\leq n})$ is a dual pair.
\end{proposition}

\begin{proof}
First of all, by Corollary \ref{cor}(2), we have that $M\in\WF_{\leq n}$ if and only if $M^+\in\WI_{\leq n}$. Moreover, from \cite[Prop. 3.3]{GW} and the commutativity of the functor Ext and the direct product in the second variable, we have that the class $\WI_{\leq n}$ is closed under direct products. In particular, $\WI_{\leq n}$ is closed under finite direct sums. Now given an exact sequence $0\rightarrow X\rightarrow Y\rightarrow Z\rightarrow 0$ with $X\in \WI_{\leq n}$. Then it is easy to verify that $Y\in \WI_{\leq n}$ if and only if $Z\in \WI_{\leq n}$ from the induced long exact sequence. Following this fact and the proof of \cite[Prop. 1.4]{Ho}, we have that $\WI_{\leq n}$ is closed under direct summands. Therefore, $(\WF_{\leq n},\WI_{\leq n})$ is a dual pair by definition.
\end{proof}

As a direct result of \cite[Thm. 3.1]{HJ09}, we have

\begin{corollary}
The class $\WF_{\leq n}$ is closed under pure submodules, pure quotients and pure extensions.
\end{corollary}

Since the functor Tor can commute with  coproducts in the second variable, it follows from \cite[Prop. 3.4]{GW} that $\WF_{\leq n}$ is closed under  coproducts, that is, the dual pair $(\WF_{\leq n},\WI_{\leq n})$ is coproduct-closed.  Following \cite[Thm. 3.1(b)]{HJ09}, we have

\begin{theorem}\label{n-wf-cover}
The class $\WF_{\leq n}$ is covering.
\end{theorem}

Now let $\{M_i\}_{i\in I}$ be a family of  $R$-modules. Choose, for each $i\in I$, a left $\WF$-resolution of $M_i$ as follows:
$$
\cdots\rightarrow W_i^2\rightarrow W_i^1\rightarrow W_i^0\rightarrow M_i\rightarrow0.
$$
Then there is an exact sequence
$$
\cdots\rightarrow \prod_{i\in I}W_i^2\rightarrow \prod_{i\in I}W_i^1\rightarrow \prod_{i\in I}W_i^0\rightarrow \prod_{i\in I}M_i\rightarrow0,
$$
which is a left $\WF$-resolution of $\prod_{i\in I}M_i$ by \cite[Thm. 2.13]{GW} and \cite[Thm. 1.2.9]{Xu}. So we can get that $$\operatorname{Ext}_n^{\mathcal{WF}}(L,\prod_{i\in I}M_i)\cong\prod_{i\in I}\operatorname{Ext}_n^{\mathcal{WF}}(L,M_i)$$ for all $R$-modules $L$ and all $n\geq 0$. Following this isomorphism and Proposition \ref{charwfd}, we have that $\WF_{\leq n}$ is closed under direct products. So, by \cite[Thm. 3.1(a)]{HJ09}, we obtain

\begin{theorem}\label{n-wf-enve}
The class $\WF_{\leq n}$ is preenveloping.
\end{theorem}

Now we turn to see the pair $(\WI_{\leq n},\WF_{\leq n})$.

\begin{proposition}\label{dual pair 2}
The pair $(\WI_{\leq n},\WF_{\leq n})$ is a dual pair.
\end{proposition}

\begin{proof}
First of all, by Corollary \ref{cor}(1), we have that $M\in\WI_{\leq n}$ if and only if $M^+\in\WF_{\leq n}$. Moreover, from \cite[Prop. 3.4]{GW} and the commutativity of the functor Tor and the direct sum, we have that the class $\WF_{\leq n}$ is closed under direct sums. Moreover, as a similar argument to the proof of Proposition \ref{dual pair 1}, we have that $\WF_{\leq n}$ is closed under direct summands. Therefore, $(\WI_{\leq n},\WF_{\leq n})$ is a dual pair by definition.
\end{proof}

As a direct result of \cite[Thm. 3.1]{HJ09}, we have

\begin{corollary}\label{cor11}
The class $\WI_{\leq n}$ is closed under pure submodules, pure quotients, and pure extensions.
\end{corollary}

Since the functor Ext can commute with direct products in second variable, it follows from \cite[Prop. 3.3]{GW} that $\WI_{\leq n}$ is closed under direct products. Following \cite[Thm. 3.1(a)]{HJ09}, we have

\begin{theorem}\label{n-wi-enve}
The class $\WI_{\leq n}$ is preenveloping.
\end{theorem}

By \cite[Prop. 2.3(1)]{GW} and \cite[Prop. 1.2.4]{Xu},  we also get that $$\operatorname{Tor}^n_{\mathcal{W}}(L,\coprod_{i\in I}M_i)\cong\coprod_{i\in I}\operatorname{Tor}^n_{\mathcal{W}}(L,M_i)$$ for any $R$-module $L$ and any $n\geq 0$. Following this isomorphism and \cite[Thm. 5.7]{GH}, we have that $\WI_{\leq n}$ is closed under coproducts. So, by \cite[Thm. 3.1(b)]{HJ09}, we obtain

\begin{theorem}\label{n-wi-cover}
The class $\WI_{\leq n}$ is covering.
\end{theorem}

As a special case of Theorems \ref{n-wf-cover}, \ref{n-wf-enve}, \ref{n-wi-enve} and \ref{n-wi-cover}, we have

\begin{corollary}(\cite[Thms. 3.1 and 3.4]{GH} and \cite[Thm. 2.15]{GW})
The classes $\WI$ and $\WF$ are covering and preenveloping respectively.
\end{corollary}

\begin{proposition}
The following are equivalent:

$(1)$ ${_RR}$ has weak injective dimension at most $n$;

$(2)$ Every right $R$-module has a monic $\WF_{\leq n}$-preenvelope;

$(3)$ Every  $R$-module has an epic $\WI_{\leq n}$-cover;

$(4)$ Every injective right $R$-module has weak flat dimension at most $n$;

$(5)$ Every projective $R$-module  has weak injective dimension at most $n$;

$(6)$ Every flat $R$-module  has weak injective dimension at most $n$.
\end{proposition}

\begin{proof}
(1) $\Rightarrow$ (2). Let $M$ be a right $R$-module. Then there is a $\WF_{\leq n}$-preenvelope $\varphi:M\rightarrow W$. Consider an exact sequence $0\rightarrow M\rightarrow \prod({_RR})^+$. Since ${_RR}$ has weak injective dimension at most $n$ by (1), then $({_RR})^+$ has weak flat dimension at most $n$ by Corollary \ref{cor}, and hence  $\prod({_RR})^+\in \WF_{\leq n}$. Now from the following commutative diagram
$$
\xymatrix@C=0.5cm{
  0 \ar[d] & \\
  M\ar[d]\ar[r]^{\varphi} & W\ar@{.>}[ld]\\
  \prod({_RR})^+ }
$$
we can get that the $\WF_{\leq n}$-preenvelope $\varphi:M\rightarrow W$ is monic.

(2) $\Rightarrow$ (4). Let $I$ be an injective right $R$-module. By (2), there is an exact sequence $0\rightarrow I\rightarrow W\rightarrow W/I\rightarrow 0$ with $W\in \WF_{\leq n}$. Moreover, this sequence is split, so $I$ belongs to $\WF_{\leq n}$ as a direct summand of $W$.

(4) $\Rightarrow$ (6). Let $F$ be a flat $R$-module. Then $F^+$ is injective by \cite[Prop. 3.54]{Ro}, and hence $F^+\in \WF_{\leq n}$ by hypothesis. Moreover, by Corollary \ref{cor}, we have  $F\in \WI_{\leq n}$.

(1) $\Rightarrow$ (3). Let $M$ be an $R$-module. Then there is a $\WI_{\leq n}$-cover $\psi:W\rightarrow M$. Consider an exact sequence $L\stackrel{f}\rightarrow M\rightarrow 0$ with $L$ free. Since ${_RR}$ has weak injective dimension at most $n$ by (1),  $L$ has weak injective dimension at most $n$. Hence, there exists a map $g:L\rightarrow W$ such that $\psi g=f$. Since $f$ is epic, we can get that $\psi:W\rightarrow M$ is also epic.

(3) $\Rightarrow$ (5).  Let $P$ be a projective  $R$-module. By (3), there is an exact sequence $0\rightarrow K\rightarrow W\rightarrow P\rightarrow 0$ with $W\in \WI_{\leq n}$. Moreover, this sequence is split, so $P$ belongs to  $\WI_{\leq n}$ as a direct summand of $W$.

(6) $\Rightarrow$ (5) and (5) $\Rightarrow$ (1) are trivial.
\end{proof}

\begin{proposition}\label{at most n+1}
The following are equivalent:

$(1)$ Every  $R$-module has a monic $\WI_{\leq n}$-cover;

$(2)$ Every right $R$-module has an epic $\WF_{\leq n}$-envelope;

$(3)$ Every  $R$-module has weak injective dimension at most $n+1$;

$(4)$ Every right $R$-module has weak flat dimension at most $n+1$;

$(5)$ Every quotient of any weak injective $R$-module is in $\WI_{\leq n}$;

$(6)$ Every submodule of any weak flat right $R$-module is in $\WF_{\leq n}$;

$(7)$ The kernel of any $\WI_{\leq n}$-precover of any $R$-module  is in $\WI_{\leq n}$;

$(8)$ The cokernel of any $\WF_{\leq n}$-preenvelope of any right $R$-module  is in $\WF_{\leq n}$;

$(9)$ $\operatorname{l.sp.gldim}(R)\leq n+1$, where $$\operatorname{l.sp.gldim}(R)=\operatorname{sup}\{\operatorname{pd}_R(M)\mid M \mbox{ is a super finitely presented  $R$-module}\}.$$
\end{proposition}

\begin{proof}
(1) $\Rightarrow$ (7). Let $M$ be an $R$-module and $f:W\rightarrow M$ a $\WI_{\leq n}$-precover. By (1), there exists a monic $\WI_{\leq n}$-precover $g:W_0\rightarrow M$. Now let $K=\mbox{Ker}f$, then $K\oplus W_0\cong W$ by \cite[Lem. 8.6.3]{EJ}. Hence $K$ is in $\WI_{\leq n}$ as a direct summand of $W$.

(7) $\Rightarrow$ (5). Let $N$ be a quotient of a weak injective $R$-module $M$. By Theorem \ref{n-wi-cover}, there is a $\WI_{\leq n}$-cover $W\rightarrow N$. Let $K=\mbox{Ker}(W\rightarrow N)$. Then we have the following commutative diagram:
$$
\xymatrix@C=0.5cm{
&&M\ar[r]\ar@{.>}[d] & N\ar[r]\ar@{=}[d]& 0\\
  0 \ar[r] & K \ar[r]^{} & W \ar[r]^{} & N  &  }
$$
In particular,  we get an exact sequence $0 \ra K \ra W \ra N  \ra 0 $. Now since $K\in \WI_{\leq n}$ by (7), and $W\in \WI_{\leq n}$,  we obtain that  $N\in \WI_{\leq n}$ from the induced long exact sequence.

(5) $\Rightarrow$ (3). Let $M$ be an $R$-module and $M\ra W^0$ a weak injective preenvelope which is monic. Let $N=\mbox{Coker}(M\ra W^0)$. Then $N$ has weak injective dimension at most $n$ by (5), that is, there exists an exact sequence $0\ra N\ra W^1\ra\cdots \ra W^{n+1}\ra 0$. By assembling these two sequences, we will have that $M$ has weak injective dimension at most $n+1$.

(3) $\Leftrightarrow$ (4) $\Leftrightarrow$ (9) follow from \cite[Thm. 3.8(2)]{GW}.

(9) $\Rightarrow$ (1). Let $M$ be in $\WI_{\leq n}$ and $N$ a quotient of $M$. Then there is an exact sequence $0\ra L\ra M\ra N\ra 0$ with $L=\mbox{Ker}(M\ra N)$. Let $F$ be any super finitely presented $R$-module. By applying the functor $\Hom_R(F,-)$ to it, we have the following exact sequence
$$
\Ext^{n+1}_R(F,M)\ra \Ext^{n+1}_R(F,N)\ra \Ext^{n+2}_R(F,L).
$$
Since $M\in \WI_{\leq n}$, $\Ext^{n+1}_R(F,M)=0$. Moreover, $\wid_R(L)\leq n+1$ by (9) and \cite[Thm. 3.8(2)]{GW}, and hence $\Ext^{n+2}_R(F,L)=0$. Therefore,  $\Ext^{n+1}_R(F,N)=0$, which means that $\wid_R(N)\leq n$, that is, $\WI_{\leq n}$ is closed under quotients. Moreover, since  $\WI_{\leq n}$ is closed under coproducts, it follows from \cite[Prop. 4]{GRT} that every  $R$-module has a monic $\WI_{\leq n}$-cover.

Dually, (2) $\Rightarrow$ (8) $\Rightarrow$ (6) $\Rightarrow$ (4) and (9) $\Rightarrow$ (2) hold.
\end{proof}

\subsection{Applications: the orthogonal class of $\WI_{\leq n}$}

Let $\mathcal{C}$ be a class of $R$-modules. We fix some notations as follows:
\begin{gather*}
  {^\bot \mathcal{C}}:=\mbox{Ker}\Ext^1_R(-,\mathcal{C})=\{M\mid \Ext^1_R(M,C)=0\mbox{ for any }C\in \mathcal{C}\}; \\
  {\mathcal{C}^\bot}:=\mbox{Ker}\Ext^1_R(\mathcal{C},-)=\{M\mid \Ext^1_R(C,M)=0\mbox{ for any }C\in \mathcal{C}\};\\
{^\top \mathcal{C}}:=\mbox{Ker}\Tor^1_R(-,\mathcal{C})=\{M\mid \Tor_1^R(M,C)=0\mbox{ for any }C\in \mathcal{C}\}.
\end{gather*}

For an $R$-module $M$, by Theorem \ref{n-wi-enve} there exists a monomorphism $M\ra W$ with $W\in \WI_{\leq n}$, and hence we get an exact sequence $0\ra M\ra W\ra N\ra 0$. If $M\in {\WI_{\leq n}}^\bot$, then it is easy to verify that $W\ra N$ is an epic $\WI_{\leq n}$-precover of $N$. In fact, we have

\begin{proposition}
For an $R$-module $M$,  $M\in {\WI_{\leq n}}^\bot$ if and only if for every short exact sequence $0\ra M\ra W\ra N\ra 0$ with $W\in \WI_{\leq n}$, $W\ra N$ is a $\WI_{\leq n}$-precover of $N$.
\end{proposition}

\begin{proof}
$\Rightarrow$ Easy.

$\Leftarrow$
Let $0\ra M\ra E(M)\ra N\ra 0$ be an exact sequence where $E(M)$ is the injective envelope of $M$ and $N=\mbox{Coker}(M\ra E(M))$. For any $R$-module $L$, we have the following exact sequence
$$
\Hom_R(L,E(M))\ra \Hom_R(L,N)\ra \Ext^1_R(L,M)\ra 0.
$$
Since $E(M)$ belongs to $\WI_{\leq n}$, by hypothesis $E(M)\ra N$ is  a $\WI_{\leq n}$-precover of $N$. Thus, if $L\in \WI_{\leq n}$, then we have an exact sequence $\Hom_R(L,E(M))\ra \Hom_R(L,N)\ra  0.$ Comparing it with the above exact sequence, we will have that $\Ext^1_R(L,M)=0$ for any $L\in \WI_{\leq n}$, that is, $M\in {\WI_{\leq n}}^\bot$.
\end{proof}

Similarly, we have

\begin{proposition}
For an $R$-module $M$, if $M\in {^\bot\WI_{\leq n}}$, then for every short exact sequence $0\ra L\ra W\ra M\ra 0$ with $W\in \WI_{\leq n}$, $L\ra W$ is a $\WI_{\leq n}$-preenvelope of $L$. The converse holds if moreover $\wid({_R R})\leq n$.
\end{proposition}

\begin{proposition}
For a right $R$-module $M$, if there is an exact sequence $0\ra L\ra F\ra M\ra 0$ with $F$ flat and $L\ra F$ a $\WF_{\leq n}$-preenvelope of $L$,   then $M\in {^\top \WI_{\leq n}}$. The converse holds if moreover $M$ is finitely presented.
\end{proposition}

\begin{proof}
Let $0\ra L\ra F\ra M\ra 0$ be an exact sequence  with $F$ flat and $L\ra F$ a $\WF_{\leq n}$-preenvelope of $L$. For any $W\in \WI_{\leq n}$, since $W^+\in \WF_{\leq n}$,  we have an exact sequence $\Hom_R(F,W^+)\ra \Hom_R(L,W^+)\ra 0$. By the adjoint isomorphism, the sequence $(F\otimes_RW)^+\ra (L\otimes_RW)^+\ra 0$ is exact, which implies $0\ra L\otimes_RW\ra F\otimes_RW$ is also exact. On the other hand, by applying the functor $-\otimes_RW$ to the previous short exact sequence, we get the following exact sequence
$$
0\ra \Tor^1_R(M,W)\ra L\otimes_RW\ra F\otimes_RW.
$$
Therefore, we have $\Tor_1^R(M,W)=0$ for any $W\in \WI_{\leq n}$, that is, $M\in {^\top \WI_{\leq n}}$.

For the converse, since $M$ is finitely presented, by definition there is an exact sequence $0\ra K\ra P\ra M\ra 0$ with $P$ finitely generated projective and $K$ finitely generated. To end the proof, it suffices to prove that $K\ra P$ is a $\WF_{\leq n}$-preenvelope, or equivalently prove that the sequence $\Hom_R(P,W)\ra \Hom_R(K,W)\ra 0$ is exact for any $W\in \WF_{\leq n}$. Indeed, for any $W\in \WF_{\leq n}$, we have $W^+\in \WI_{\leq n}$, and hence $\Tor_1^R(M,W^+)=0$ by hypothesis. This implies an exact sequence $0\ra K\otimes_RW^+\ra P\otimes_RW^+$ and a commutative diagram as follows:
$$
\xymatrix{
0\ar[r] & K\otimes_RW^+\ar[r]\ar[d]^{\kappa_K}& P\otimes_RW^+\ar[d]^{\kappa_P}\\
&\Hom_R(K,W)^+\ar[r]&\Hom_R(P,W)^+
}
$$
By \cite[Lem. 3.55]{Ro}, $\kappa_P$ is an isomorphism. Now since  $K$ is finitely generated, there is an exact sequence $P'\ra K\ra 0$ with $P'$  finitely generated projective. Consider the following commutative diagram:
$$
\xymatrix{
 P'\otimes_RW^+\ar[r]\ar[d]^{\kappa_{P'}}& K\otimes_RW^+\ar[d]^{\kappa_K}\ar[r]&0\\
\Hom_R(P',W)^+\ar[r]&\Hom_R(K,W)^+\ar[r] &0
}
$$
Since ${\kappa_{P'}}$ is an isomorphism, ${\kappa_K}$ is an epimorphism. Therefore, from the first commutative diagram we get that $0\ra\Hom_R(P,W)\ra \Hom_R(K,W)$ is exact, as desired.
\end{proof}

 We denote by  $\mathcal{I}$ and $\mathcal{F}$   the classes of injective and flat $R$(resp. right $R$)-modules respectively.

\begin{proposition}
$(1)$ $\mathcal{I}={\WI_{\leq n}}^\bot \bigcap \mathcal{WI}_{\leq n+1}$;
$(2)$  $\mathcal{F}={^\top\WI_{\leq n}} \bigcap \mathcal{WF}_{\leq n+1}$.
\end{proposition}

\begin{proof}
(1) Clearly, $\mathcal{I}\subseteq{\WI_{\leq n}}^\bot \bigcap \mathcal{WI}_{\leq n+1}$. For any $M\in {\WI_{\leq n}}^\bot \bigcap \mathcal{WI}_{\leq n+1}$, consider an exact sequence $0\ra M\ra E(M)\ra E(M)/M\ra 0$, where $E(M)$ is the injective envelope of $M$. Since $M\in \mathcal{WI}_{\leq n+1}$, we have $E(M)/M\in \mathcal{WI}_{\leq n}$. Moreover, $M\in {\WI_{\leq n}}^\bot$ implies $\Ext^1_R(E(M)/M,M)=0$. Hence $M$ is injective as a direct summand of $E(M)$, that is, ${\WI_{\leq n}}^\bot \bigcap \mathcal{WI}_{\leq n+1}\subseteq \mathcal{I}$. Therefore, $\mathcal{I}={\WI_{\leq n}}^\bot \bigcap \mathcal{WI}_{\leq n+1}$.

(2) Clearly, $\mathcal{F}\subseteq{^\top\WI_{\leq n}} \bigcap \mathcal{WF}_{\leq n+1}$. For any $M\in {^\top\WI_{\leq n}} \bigcap \mathcal{WF}_{\leq n+1}$, we will get $M^+\in {\WI_{\leq n}}^\bot$ from the isomorphism $\Ext^1_R(L,M^+)\cong \Tor^R_1(M,L)^+$ (\cite[Lem. 1.2.11]{GT}). Moreover, since $M\in \mathcal{WF}_{\leq n+1}$, we have $M^+\in \mathcal{WI}_{\leq n+1}$ by Corollary \ref{cor}, and hence $M^+\in \mathcal{I}$ by (1). This implies that $M$ is flat. Thus  $\mathcal{F}={^\top\WI_{\leq n}} \bigcap \mathcal{WF}_{\leq n+1}$.
\end{proof}

The cotorsion theory plays a crucial role in the study of the existence of covers and envelopes. Recall that a pair $(\mathcal{C},\mathcal{D})$ of classes of $R$-modules is called a \emph{cotorsion pair} if $\mathcal{C}^\bot=\mathcal{D}$ and ${^\bot\mathcal{D}}=\mathcal{C}$. A cotorsion pair $(\mathcal{C},\mathcal{D})$ is said to be \emph{perfect} if  every $R$-module admits a $\mathcal{C}$-cover and a $\mathcal{D}$-envelope. A cotorsion pair $(\mathcal{C},\mathcal{D})$ is said to be \emph{hereditary} if whenever $0\ra X\ra Y\ra Z\ra 0$ is exact with $Y,Z\in \mathcal{C}$, then $X\in \mathcal{C}$, or equivalently, if whenever $0\ra X\ra Y\ra Z\ra 0$ is exact with $X,Y\in \mathcal{D}$, then $Z\in \mathcal{D}$.

We have known that the class ${\WI_{\leq n}}$ is covering by Theorem \ref{n-wi-cover}. In the following, we will prove that  ${\WI_{\leq n}}^\bot$ is enveloping if $\wid({_RR})\leq n$.

\begin{proposition}
 If $\wid({_RR})\leq n$, then $({\WI_{\leq n}}, {\WI_{\leq n}}^\bot)$ is a perfect cotorsion pair.
\end{proposition}

\begin{proof}
By Corollary \ref{cor11}, ${\WI_{\leq n}}$ is closed under pure submodules and pure quotients. Moreover, from the previous argument, we have that ${\WI_{\leq n}}$ is closed under extensions and coproducts.  Since $\wid({_RR})\leq n$, we also have that the class ${\WI_{\leq n}}$ contains the ground ring $R$. Thus  $({\WI_{\leq n}}, {\WI_{\leq n}}^\bot)$ is a perfect cotorsion pair by \cite[Thm. 3.4]{HJ08}.
\end{proof}

Similarly, since ${\WF_{\leq n}}$ is closed  under  extensions, coproducts,  pure submodules and pure quotients, and contains  the ground ring $R$, we can obtain

\begin{proposition}
 $({\WF_{\leq n}}, {\WF_{\leq n}}^\bot)$ is a perfect and hereditary cotorsion pair.
\end{proposition}

\begin{proof}
The hereditary property follows from the fact that if there is an exact sequence $0\ra X\ra Y\ra Z\ra 0$ with $Y,Z\in {\WF_{\leq n}}$, then $X\in {\WF_{\leq n}}$.
\end{proof}

\subsection{More descriptions on the case n=0}

Recall that a $\mathcal {C}$-envelope $\varphi: M
\rightarrow C$ of $M$ is said to have the \emph{unique mapping property} \cite{Di} if the morphism $g$ in the  diagram (\ref{cover diag}) of Definition \ref{def-env} is unique; and a monomorphism  $\varphi: M
\rightarrow C$ with $C\in \mathcal {C}$ is said to be a \emph{special $\mathcal{C}$-preenvelope} \cite{EJ} of $M$ if $\mbox{Coker}\varphi\in {^\perp\mathcal{C}}$. Dually, one has the notions of  $\mathcal {C}$-covers with {unique mapping property} and special $\mathcal{C}$-preenvelopes of a module.

Let
$$
^\perp{\mathcal{WI}}=\{M\mid \mbox{Ext}^1_R(M,W)=0 \mbox{ for any } W\in \mathcal{WI}\}.
$$
Following \cite[Def. 7.1.2]{EJ}, it is easy to verify that $(^\perp{\mathcal{WI}},\mathcal{WI})$ is a cotorsion theory which is cogenerated
by the representative set of all super finitely presented  $R$-modules.
So, by \cite[Thm. 7.4.1 and Def. 7.1.5]{EJ}, every  $R$-module $M$ has a special weak injective preenvelope, that is,
 there is an exact sequence $0\rightarrow M\rightarrow W\rightarrow V\rightarrow 0$ with $W$ weak injective and $V\in {^\perp{\mathcal{WI}}}$. Meanwhile, every $R$-module $M$ has a special $^\perp{\mathcal{WI}}$-precover, that is,  there is an exact sequence $0\rightarrow L\rightarrow Q\rightarrow M\rightarrow 0$ with $Q\in {^\perp{\mathcal{WI}}}$ and $L$ weak injective. Moreover, by Wakamatsu's Lemma (\cite[Props. 7.2.3 and 7.2.4]{EJ}), if $\epsilon:M\rightarrow W$ is a weak injective envelope of $M$, then $\mbox{Coker}\epsilon$ is in  $^\perp{\mathcal{WI}}$; and if $\rho:Q\rightarrow M$ is a $^\perp{\mathcal{WI}}$-cover, then $\mbox{Ker}\rho$ is weak injective.

\begin{proposition}\label{11111}
The following are equivalent:

$(1)$ Every  $R$-module has a weak injective cover with the unique mapping property;

$(2)$ $\operatorname{l.sp.gldim}(R)\leq 2$.
\end{proposition}

\begin{proof}
(1) $\Rightarrow$ (2). Let $M$ be any  $R$-module. From the previous argument, we have the following two exact sequences:
$$
  0\rightarrow M\stackrel{\epsilon^0}\rightarrow W^0\stackrel{\nu^0}\rightarrow V^1\rightarrow 0, \
  0\rightarrow V^1\stackrel{\epsilon^1}\rightarrow W^1\stackrel{\nu^1}\rightarrow V^2\rightarrow 0
$$
where ${\epsilon^0}$ and ${\epsilon^1}$ are special weak injective preenvelopes of $M$ and $V^1$, respectively, and thus $V^1, V^2\in{^\perp{\mathcal{WI}}}$. Assembling these two sequences, we get the following commutative diagram:
$$
\xymatrix{
0\ar[r]&M\ar[r]^{\epsilon^0}&W^0\ar[rr]^{\epsilon^1\nu^0}\ar[rd]^{\nu^0}&&W^1\ar[r]^{\nu^1}&V^2\ar[r]&0\\
&&&V^1\ar[ru]^{\epsilon^1}\ar[rd]&&&\\
&&0\ar[ru]&&0&&
}
$$
where the top row is exact. In the following, we will prove that $V^2$ is weak injective.

Let $\varepsilon:W\rightarrow V^2$ be a weak injective cover with the unique mapping property. Then there exists $\delta:W^1\rightarrow W$ such that $\nu^1=\varepsilon\delta$. So $\varepsilon\delta\epsilon^1\nu^0=\nu^1\epsilon^1\nu^0=0$, and thus $\delta\epsilon^1\nu^0=0$. This shows that $\mbox{Ker}\delta\supseteq\mbox{Im}\epsilon^1\nu^0=\mbox{Ker}\nu^1$. Consequently, there exists $\eta:V^2\rightarrow W$ such that $\delta=\eta\nu^1$, and hence we have the following commutative diagram:
$$
\xymatrix{
&&&&W\ar@<1ex>[d]^\varepsilon&\\
0\ar[r]&M\ar[r]^{\epsilon^0}&W^0\ar[r]^{\epsilon^1\nu^0}&W^1\ar[r]^{\nu^1}\ar@{.>}[ru]^\delta&V^2\ar[r]\ar@{.>}[u]^\eta&0
}
$$
Since $\varepsilon\eta\nu^1=\varepsilon\delta=\nu^1$ and $\nu^1$ is epic, we have $\varepsilon\eta=1_{V^2}$. Thus $V^2$ is weak injective as a direct summand of $W$, and hence $\wid_R(M)\leq 2$. Therefore, $\operatorname{l.sp.gldim}(R)\leq 2$ by \cite[Thm. 3.8]{GW}.

(2) $\Rightarrow$ (1). First of all, every  $R$-module has a weak injective cover by \cite[Thm. 3.1]{GH}. Choose a weak injective cover $\epsilon:W\rightarrow M$ of $M$. To end the proof, we only need to show that, for any weak injective  $R$-module $Q$ and any map $\delta:Q\rightarrow W$, $\epsilon\delta=0$ implies $\delta=0$. Indeed, since $\mbox{Im}\delta\subseteq\mbox{Ker}\epsilon$, there exists $\eta:\mbox{Coker}\delta\rightarrow M$ such that $\eta\pi=\epsilon$, where $\pi:W\rightarrow \mbox{Coker}\delta$ is the canonical projection. Note that $\wid_R(\mbox{Ker}\delta)\leq 2$  since $\operatorname{l.sp.gldim}(R)\leq 2$, and thus $\mbox{Coker}\delta$ is weak injective. It follows then that there exists $q:\mbox{Coker}\delta\rightarrow W$ such that $\eta=\epsilon q$, that is, there is the following commutative diagram:
$$
\xymatrix{
0\ar[r]&\mbox{Ker}\delta \ar[r]^i &Q\ar[r]^\delta\ar[rd]_0 &W\ar@<1ex>[r]^\pi\ar[d]^\epsilon&\mbox{Coker}\delta \ar@{.>}[l]^q\ar[r]\ar@{.>}[ld]^\eta& 0\\
&&&M&&
}
$$
Since $\epsilon (q\pi)=(\epsilon q)\pi=\eta\pi=\epsilon$ and $\epsilon$ is a cover, $q\pi$ is an isomorphism, and thus $\pi$ is monic. Therefore, $\delta=0$, as desired.
\end{proof}

As  is known to all, every module has a weak injective preenvelope, however,
weak injective envelopes may not exist in general. Now we give a  sufficient condition to ensure the existence of weak injective envelopes.

\begin{proposition}
If ${^\perp\WI}$ is closed under pure quotient, then $\WI$ is enveloping. In this case, ${^\perp\WI}$ is covering.
\end{proposition}

\begin{proof}
Clearly, the class ${^\perp\WI}$ is closed under extensions and direct sums. We will argue that ${^\perp\WI}$ is closed under pure submodules under assumption. Given a pure exact sequence $0\rightarrow L\rightarrow M\rightarrow N\rightarrow 0$ with $M\in{^\perp\WI}$. By assumption, $N\in{^\perp\WI}$, i.e. $\Ext^1_R(N,W)=0$ for any $W\in\WI$. Now, for any $W\in\WI$, consider the induced exact sequence
$$
\Ext^1_R(N,W)\rightarrow \Ext^1_R(M,W)\rightarrow \Ext^1_R(L,W)\rightarrow \Ext^2_R(N,W).
$$
We claim that $\Ext^2_R(N,W)=0$. Indeed, consider an exact sequence $0\rightarrow W\rightarrow E\rightarrow W'\rightarrow 0$ with $E$ injective. In fact, $W'$ is also weak injective from \cite[Prop. 3.1]{GW} and the long exact sequence which is induced by applying $\operatorname{Hom}_R(F,-)$, where $F$ is any super finitely presented $R$-module. Now applying the functor $\operatorname{Hom}_R(N,-)$ to it, we have the following exact sequence
$$
\Ext^1_R(N,E)\rightarrow \Ext^1_R(N,W')\rightarrow \Ext^2_R(N,W)\rightarrow \Ext^2_R(N,E),
$$
which implies that $\Ext^2_R(N,W)\cong \Ext^1_R(N,W')=0$ since $W'\in \WI$.

Now, from the previous long exact sequence, we get that $\Ext^1_R(L,W)\cong  \Ext^1_R(M,W)$. Therefore, $M\in{^\perp\WI}$ implies $L\in{^\perp\WI}$, as desired.
Moreover, since ${^\perp\WI}$ contains $R$, it follows from \cite[Thm. 3.4]{HJ08} that ${^\perp\WI}$ is covering and $\WI=({^\perp\WI})^\perp$ is enveloping.
\end{proof}

Under the
assumption of the existence of  weak injective envelopes,  we will have
the following conclusions.

\begin{proposition}\label{11}
Let $N$ be a weak injective  $R$-module, and $M\leqslant N$. If $M$ has a weak injective envelope, then the following are equivalent:

$(1)$ The inclusion $i:M\hookrightarrow N$ is a weak injective envelope of $M$;

$(2)$ The quotient module $N/M$ is in $^\perp \mathcal{WI}$, and there are no direct summands $N_1$ of $N$ with $N_1\neq N$ and $M\leqslant N_1$;

$(3)$ The quotient module $N/M$ is in $^\perp \mathcal{WI}$, and for any epimorphism $\varphi:N/M\rightarrow Q$ such that $\varphi \pi$ is a retraction, where $\pi:N\rightarrow N/M$ is the canonical map, we have $Q=0$;

$(4)$ The quotient module $N/M$ is in $^\perp \mathcal{WI}$, and any endomorphism $\tau$ of $N$ such that $\tau i=i$ is a monomorphism;

$(5)$ The quotient module $N/M$ is in $^\perp \mathcal{WI}$, and there are no nonzero submodules $N'$ of $N$ such that $M\cap N'=0$ and $N/(M\oplus N')$ is  in $^\perp \mathcal{WI}$.
\end{proposition}

\begin{proof}
(1) $\Rightarrow$ (2) follows from \cite[Prop. 7.2.4]{EJ} and \cite[Cor. 1.2.3]{Xu}.

(2) $\Rightarrow$ (1). Since  $N/M$ is in $^\perp \mathcal{WI}$, it is easy to verify that $i:M\hookrightarrow N$ is a weak injective preenvelope of $M$. Moreover, it is a weak injective envelope of $M$ by \cite[Cor. 1.2.3]{Xu} and hypothesis.

(2) $\Rightarrow$ (3). Since $\varphi \pi$ is a retraction, there exists a monomorphism $\psi$ such that $N=\mbox{Ker}(\varphi \pi)\oplus\psi(Q)$. Note that $M\leqslant\mbox{Ker}(\varphi \pi)$, so $M=\mbox{Ker}(\varphi \pi)$ by assumption. Hence $\psi(Q)=0$, and thus $Q=0$.

(3) $\Rightarrow$ (2). Assume that $N=N_1\oplus Q$ such that $M\leqslant N_1$. Let $p:N\rightarrow Q$ be the canonical projection. Then $M\leqslant \mbox{Ker}p$, and thus there exists $\gamma:N/M\rightarrow Q$ such that $\gamma\pi=p$. So $Q=0$ by hypothesis. Therefore, $N=N_1$, as desired.

(1) $\Rightarrow$ (4) is trivial.

(4) $\Rightarrow$ (1). Since  $N/M$ is in $^\perp \mathcal{WI}$, $i:M\hookrightarrow N$ is a special weak injective preenvelope of $M$. Let $\epsilon:M\rightarrow W$ be any weak injective envelope of $M$. Then there exist two maps $\delta:N\rightarrow W$ and $\sigma:W\rightarrow N$ such that the following diagrams commute:
$$
\xymatrix{
M\ar[r]^i\ar[rd]^\epsilon&N\ar@{.>}[d]^\delta\\
&W
}\ \ \ \ \ \mbox{ and }\ \ \ \
\xymatrix{
M\ar[r]^\epsilon\ar[rd]^i&W\ar@{.>}[d]^\sigma\\
&N
}
$$
that is, $\epsilon=\delta i$ and $i=\sigma\epsilon$. So $\epsilon=\delta\sigma\epsilon$ and $i=\sigma\delta i$, and thus $\delta\sigma$ is an isomorphism, which shows that $\delta$ is epic. On the other hand, $\sigma\delta$ is monic by assumption, which shows that $\delta$ is monic. Therefore, $\delta$ is an isomorphism, and hence $i$ is a weak injective envelope of $M$.

(1) $\Rightarrow$ (5). Assume that there is a nonzero submodule $N'$ of $N$ such that $M\cap N'=0$ and $N/(M\oplus N')$ is  in $^\perp \mathcal{WI}$. Let $\pi:N\rightarrow N/N'$ be the canonical map. Since $N/(M\oplus N')$ is  in $^\perp \mathcal{WI}$ and $N$ is weak injective, there exists $\theta:N/N'\rightarrow N$ such that the following diagram commute:
$$
\xymatrix{
&&N\ar[rd]^\pi&&&\\
0\ar[r]&M\ar[ru]^i\ar[rr]^{\pi i}\ar[rd]^i&&N/N'\ar[r]\ar@{.>}[ld]^\theta&N/(N'\oplus M)\ar[r]&0\\
&&N&&&
}
$$
that is, $\theta\pi i=i$. Moreover, since $i$ is a weak injective envelope, $\theta\pi$ is an isomorphism, and hence $\pi$ is an isomorphism. But this is a contradiction with $\pi(N')=0$. So (5) holds.

(5) $\Rightarrow$ (1). Let $i:M\rightarrow N$ be the canonical inclusion. Since $N/M$ is in $^\perp \mathcal{WI}$, the map $i$ is a special weak injective preenvelope. Choose a weak injective envelope $\epsilon:M\rightarrow W$ of $M$. Then we have the following commutative diagram:
$$
\xymatrix{
0\ar[r]&M\ar[r]^\epsilon\ar[rd]_i&W\ar[r]^\pi\ar@{.>}[d]^f&C\ar[r]&0\\
&&N\ar@<1ex>@{.>}[u]^g\ar[ru]_{\alpha=\pi g}&&
}
$$
that is, $f\epsilon=i$, and $g i=\epsilon$. So $g f\epsilon=\epsilon$. Note that $g f$ is an isomorphism since $\epsilon$ is an envelope. Without loss of generality, we may assume that $g f=1$. Let $\alpha=\pi g$. Then $\alpha$ is epic and $M\cap \mbox{Ker}g=0$. We claim that $M\oplus \mbox{Ker}g=\mbox{Ker}\alpha$. Indeed, it is obvious that $M\oplus \mbox{Ker}g\subseteq\mbox{Ker}\alpha$. Conversely, let $a\in \mbox{Ker}\alpha$. Then $0=\alpha(a)=\pi g(a)$, that is, $g(a)\in \mbox{Ker}\pi$. Since $\mbox{Ker}\pi=\mbox{Im}\epsilon$, there exists some $b\in M$ such that $\epsilon(b)=g(a)$. So $fg(a)=f\epsilon(b)=i(b)=b$, $g(a)=gfg(a)=g(fg(a))=g(b)$. Thus $a\in M\oplus\mbox{Ker}g$, and hence $\mbox{Ker}\alpha\subseteq M\oplus \mbox{Ker}g$. Consequently, $M\oplus \mbox{Ker}g=\mbox{Ker}\alpha$.

Moreover, $C$ is in $^\perp \mathcal{WI}$ by \cite[Prop. 7.2.4]{EJ}, and $N/(M\oplus \mbox{Ker}g)= N/\mbox{Ker}\alpha\cong C$. By assumption, $\mbox{Ker}g=0$. Thus $g$ is an isomorphism. Therefore, $i$ is a weak injective envelope of $M$.
\end{proof}

Recall from \cite{AF} that a \emph{minimal injective extension} of an $R$-module $M$ is a monomorphism $i:M\rightarrow E$ with $E$ injective such that for every  monomorphism $\varphi:M\rightarrow Q$ with $Q$ injective, there exists a monomorphism $\phi:E\rightarrow Q$ such that $\varphi=\phi i$.

Similarly, we define the \emph{minimal weak injective extension} of modules as follows:  A minimal weak injective extension of an $R$-module $M$ is a monomorphism $i:M\rightarrow W$ with $W$ weak injective such that for every  monomorphism $\varphi:M\rightarrow Q$ with $Q$ weak injective, there exists a monomorphism $\phi:W\rightarrow Q$ such that $\varphi=\phi i$.

\begin{corollary}
Let $M$ be a weak injective  $R$-module, $M'$ a submodule of $M$ such that $M/M'$ is in $^\perp \mathcal{WI}$ and $M'$ has a weak injective envelope.

$(1)$ If $M'$ is an essential submodule of $M$, then the inclusion $i:M'\rightarrow M$ is a weak injective envelope of $M'$.

$(2)$ If  $i:M'\rightarrow M$ is a minimal weak injective extension of $M'$, then $i$ is a weak injective envelope of $M'$.
\end{corollary}

\begin{proof}
(1) follows from Proposition $\ref{11}$.

(2) Let $\varepsilon:M'\rightarrow E(M')$ be an injective envelope of $M'$. Since $i:M'\rightarrow M$ is a minimal weak injective extension of $M'$, by  definition there exists $\delta:M\rightarrow E(M')$ such that $\delta i=\varepsilon$, that is, there is the following commutative diagram:
$$
\xymatrix{
  M' \ar[dr]_{i} \ar[rr]^{\varepsilon} &&   E(M')     \\
 & M \ar@{.>}[ur]_{\delta}    &                 }
$$
Note that the injective envelope $\varepsilon$ of $M'$ is in fact an essential weak injective extension of $M'$. As a similar argument to the proof of \cite[Cor. 18.11]{AF}, we have that $i$ is also essential, that is, $M'$
 is an essential submodule of $M$. Thus $i$ is a weak injective envelope of $M$ by (1).\end{proof}

\begin{proposition}
If an $R$-module $M$ has a $^\perp \mathcal{WI}$-cover, then $M$ has a special weak injective preenvelope $\eta:M\rightarrow W$ such that $W$ has a $^\perp \mathcal{WI}$-cover.
\end{proposition}

\begin{proof}
Assume that $\theta:Q\rightarrow M$ is a $^\perp \mathcal{WI}$-cover of $M$. Let $K=\mbox{Ker}\theta$. Then $K$ is weak injective by \cite[Cor. 7.2.3]{EJ}. Note that $Q$ has a special weak injective preenvelope, that is, there is an exact sequence $0\rightarrow Q\stackrel{f}\rightarrow W\stackrel{g}\rightarrow L\rightarrow 0$ with $W$ weak injective and $L \in {^\perp \mathcal{WI}}$.

Consider the following push-out diagram:
$$
\xymatrix{
&&0\ar[d]&0\ar[d]&\\
0\ar[r]&K\ar[r]\ar@{=}[d]&Q\ar[r]^\theta\ar[d]^f&M\ar[r]\ar[d]^\eta&0\\
0\ar[r]&K\ar[r]&W\ar[r]^\xi\ar[d]^g&N\ar[r]\ar[d]&0\\
&&L\ar@{=}[r]\ar[d]&L\ar[d]&\\
&&0&0&
}
$$
It follows easily from the third row in the above diagram that $N$ is weak injective. So $\eta:M\rightarrow N$ is a special weak injective preenvelope of $M$. Moreover, it is easy to verify that $W$ is in $^\perp \mathcal{WI}$ from the third column in the above diagram. Thus $\xi:W\rightarrow N$ is a special $^\perp \mathcal{WI}$-precover of $N$. In the following, we will show that $\xi:W\rightarrow N$ is in fact a  $^\perp \mathcal{WI}$-cover of $N$.

Let $\gamma$ be an endomorphism of $W$ with $\xi\gamma=\xi$. Then $\xi(\gamma f)=(\xi\gamma) f=\xi f=\eta\theta$. Note that the upper right-hand square of the above diagram is bicartesian, in particular, it is  a pull-back. So there exists $h:Q\rightarrow Q$ such that $\theta h=\theta$ and $fh=\gamma f$, that is, there is the following commutative diagram:
$$
\xymatrix{
  Q \ar@/_/[ddr]_{\gamma f} \ar@/^/[drr]^{\theta}
    \ar@{.>}[dr]|-{h}                   \\
   & Q \ar[d]^{f} \ar[r]_{\theta}
                      & M \ar[d]_{\eta}    \\
   & D \ar[r]^{\xi}     & W               }
$$
Since $\theta$ is a cover, $h$ is an isomorphism by the equality $\theta h=\theta$. Let $\gamma(w')=0$ for some  $w'\in W$. Then $\xi(w')=\xi\gamma(w')=0$, and so $g(w')=0$. Thus $w'\in \mbox{Im}f$, that is, there exists some $q\in Q$ such that $w'=f(q)$. So $fh(q)=\gamma f(q)=\gamma(w')=0$, and hence $q=0$ since $fh$ is monic. Therefore, $w'=0$, which shows that $\gamma$ is monic. On the other hand, since $\xi\gamma(w)=\xi(w)$ for any $w\in W$, that is, $w-\gamma(w)\in \mbox{Ker}\xi (\subseteq \mbox{Ker}g)$, there exists some $q'\in Q$ such that $f(q')=w-\gamma(w)$. Thus $w=f(q')+\gamma(w)=\gamma f h^{-1}(q')+\gamma(w)=\gamma(f h^{-1}(q')+w)$, and hence $\gamma$ is epic. Therefore, $\gamma$ is an isomorphism, as desired.
\end{proof}

\section*{Acknowledgements}

The author would like to thank the referees for their careful reading and valuable remarks that improved the presentation of this work. He also thanks Professor Zhaoyong Huang for his encouragement. This research was partially supported by NSFC (Grant No. 11571164) and a Project Funded
by the Priority Academic Program Development of Jiangsu Higher Education Institutions.

\end{document}